\documentclass[letterpaper, 10 pt, conference]{ieeeconf}  
\usepackage{amsmath}
\usepackage{amssymb}
\IEEEoverridecommandlockouts                              
\usepackage{graphics} 
\usepackage{epsfig} 
\usepackage{epstopdf}
\usepackage{mathptmx} 
\usepackage{times} 
\usepackage{amsmath} 
\usepackage{amssymb}  
\usepackage[linesnumbered,lined,boxed,commentsnumbered]{algorithm2e}
\usepackage[noend]{algpseudocode}
\usepackage{subfigure}

\newtheorem{mydef}{Definition}
\newtheorem{mythm}{Theorem}

\newtheorem{mylem}{Lemma}

\newtheorem{myass}{Assumption}

\newtheorem{cor}{Corollary}

\newcommand{\R}{\mathbb{R}}

\newcommand{\N}{\mathbb{N}}

\newcommand{\norm}[1]{\lVert{#1}\rVert}

\title{\LARGE \bf
Estimating the Region of Attraction Using Polynomial Optimization:  a
Converse Lyapunov Result
}

\author{Hesameddin Mohammadi, Matthew M. Peet
}

\begin{document}

\maketitle
\thispagestyle{empty}
\pagestyle{empty}
\begin{abstract}
In this paper, we propose an iterative method for using SOS programming to estimate the region of attraction of a polynomial vector field, the conjectured convergence of which necessitates the existence of polynomial Lyapunov functions whose sublevel sets approximate the true region of attraction arbitrarily well. The main technical result of the paper is the proof of existence of such a Lyapunov function. Specifically, we use the Hausdorff distance metric to analyze convergence and in the main theorem demonstrate that the existence of an $n$-times continuously differentiable maximal Lyapunov function implies that for any $\epsilon>0$, there exists a polynomial Lyapunov function and associated sub-level set which together prove stability of a set which is within $\epsilon$ Hausdorff distance of the true region of attraction. The proposed iterative method and probably convergence is illustrated with a numerical example.
\end{abstract}

\section{Introduction}
In this paper we consider the problem of estimating the region of attraction of systems of nonlinear Ordinary Differential Equations (ODE) of the form
\begin{align}
\label{eq: ODE introduction}
\dot{x}=f(x),\quad\,x(0) = x_0,
\end{align}
where $ f:\mathbb{R}^n\rightarrow\mathbb{R}^n $ is the vector field and $ x_0\in\mathbb{R}^n $ is the initial condition. If we define $ g(x,t) $ as the associated solution map and suppose $f(0)=0$, then the Region of Attraction (ROA) is defined as
\begin{align}
\label{eq: ROA introduction }
S_f:=\{x\in\mathbb{R}^n: \underset{t\rightarrow\infty}{\lim}g(x,t) = 0  \}.
\end{align}
Accurate estimates of the ROA are necessary for such problems as, e.g. flight control verification and validation~\cite{Chakraborty2011335} and determining the range of concentration over which a biological system takes on a certain set of steady state concentrations corresponding to a preferred phenotype~\cite{6377982}.

Existing approaches to approximating the ROA can be divided into  Lyapunov and non-Lyapunov based categories. Among the non-Lyapunov based approaches, we find the use of occupation measures to outer-approximate the ROA of polynomial vector fields as in \cite{6606873}. Unfortunately however,  these outer-approximations are not themselves stable and furthermore the method is restricted to certain classes of vector field. Some non-Lyapunov based approaches exist which, although not categorized as Lyapunov, still make use of Lyapunov based arguments; examples include a trajectory reversing method by backward integration of the vector field for a number of stable initial conditions as in~\cite{doi:10.1080/00207179508921938} and advecting a stable sub-level set of a polynomial backward in time as introduced in~\cite{6517253}. However, both these approaches require an initial stable set which is typically obtained using Lyapunov based methods.

By constrast, almost all Lyapunov-based methods for estimating the ROA are based on the search for a Lyapunov function $ V(x) $ and for a positive scaler $ b $ such that \vspace{0.00 in}$ \dot{V} (x)$ is negative over the sub-level set $ C:=\{x:V(x)\le b\} $~\cite{Hahn19671}. Given such $V$ and $b$, it can be shown that the connected component of $C$ containing the equilibrium is an inner-approximation to the ROA. Neglecting accurate estimates of the ROA momentarily, if we are interested in establishing the existence of a Lyapunov function which is decreasing over some bounded set, then the problem is convex and for a polynomial vector field, there are a number of recent results which use convex optimization-based approaches to search for a polynomial Lyapunov function. See, e.g. ~\cite{912314} and~\cite{1184414} or the Sum-of-Squares based open source toolboxes for constructing polynomial Lyapunov functions in SOSTOOLS~\cite{sostools} and Yalmip~\cite{1393890}. Alternatives to the SOS approach can be found in~\cite{Kamyar20152383}.

If we return to the problem of estimating the region of attraction, however, then the problem of searching for a polynomial LF with maximal sublevel sets is a bilinear Sum of Squares (SOS) program as shown in. e.g.~\cite{4471858} and~\cite{Topcu20082669}. To deal with this bilinearity, researchers have turned to Genetic Algorithms and fuzzy modeling, examples of which can be found in~\cite{6031233} and~\cite{6878443}, respectively. Extensions to nonlinear systems with uncertainties can also be found in~\cite{Zec2010445} and~\cite{5337881}. One approach to overcoming this bilinearity, as proposed in Section~\ref{sec: Numerical Example}, is to increase the diameter of the region on which the Lyapunov function is decreasing. As the region approaches the true ROA, the problem approaches infeasibility. We have conjectured that this asymptotic infeasibility is then helpful, in that it implicitly constrains the polynomial Lyapunov function to approximate a maximal Lyapunov function and hence approximate its sublevel sets and hence provide asymptotically accurate estimates of the region of attraction. However, this conjecture is purely speculative and, in fact, is based on the assumption that polynomial Lyapunov function can estimate the domain of attraction arbitrarily well. In this paper, we examine the assumption that polynomial Lyapunov functions can estimate the domain of attraction arbitrarily well and show that, in fact, polynomial Lyapunov functions can estimate the ROA as well as continuously differentiable maximal Lyapunov functions. 

Results on the existence of Lyapunov functions establishing stability of ODEs and estimates of the ROA are classified as converse Lyapunov theorems~\cite{Hahn19672}. Among the class of converse Lyapunov results, there are two sub-types we will use in this paper and are discussed in Section~\ref{sec :Problem Statement}. The first , Massera-type~\cite{Hahn19672} establishes existence of smooth Lyapunov functions on bounded subsets of the ROA with quantitative upper and lower bounds. The second, maximal-type, establish Lyapunov functions which approach infinity at the boundary of the ROA and whose level sets form asymptotically accurate inner-approximations to the ROA~\cite{VANNELLI198569}. However, despite extensive literature on converse Lyapunov theory, there are very few results on the existence of polynomial Lyapunov functions and in particular, it has never been shown that for any desired accuracy $ \epsilon>0 $, there exists a polynomial LF for which a sub-level set is within the Hausdorff $ \epsilon $-distance of the ROA. Furthermore, there is reason to doubt the existence of such polynomial Lyapunov functions. For instance, in~\cite{6161499} we find an example of a globally asymptotically stable polynomial system for which there exists no polynomial LF proving the global stability of the system. These counter examples motivate us to investigate the conditions under which the sub-level sets of polynomial LF can accurately approximate the ROA of nonlinear systems.
%
%
%
%
%
%
%
%

The goal of this paper, then, is to resolve the problem of whether the sub-level sets of polynomial Lyapunov functions  can provide arbitrarily accurate inner-approximations to the ROA. For this purpose, we will combine the approximation results on the existence of polynomial LF from~\cite{4908942} and the idea of maximal LF introduced in~\cite{VANNELLI198569}  to propose sufficient conditions that guarantee the sub-level sets of polynomial LF provide inner-approximations to the ROA arbitrarily well in the sense of Hausdorff distance.
The first of these results, as shown in~\cite{4908942}, is that  polynomials can approximate sufficiently smooth functions with a point-wise weight on the error given by $ 1/x^Tx $ in the Sobolov space and as an important implication, for any  sufficiently smooth LF proving exponential decay on a bounded region, there exists a polynomial LF on the same region which proves the same exponential decay over a neighborhood $ U$ of the equilibrium.
The second result, as stated in~\cite{VANNELLI198569}, is that under mild conditions on the vector field, there exists a maximal LF $ V:S\rightarrow R $ satisfying
 \begin{align*}
 \underset{x\rightarrow y}{\lim} V(x) = \infty,\quad\forall y\in\partial S_f,
 \end{align*}
 where $ \partial S $ denotes the boundary of the region of attraction $ S_f $ given in Eq.~\eqref{eq: ROA introduction }.
 
The main result of this paper, then, and as stated in Theorem~\ref{thm: main}, is that if $ S_f $  is bounded and there exists a sufficiently smooth maximal LF $ V(x) $, which is defined over $ S_f $ and proves exponential decay over each of the sub-level sets
\[
\{x:V(x)\le\alpha\},\;\;\alpha\in[0,+\infty),
\]
 then for any desired accuracy $ \epsilon>0 $, there exists a polynomial LF $ P_\epsilon(x) $ with a sub-level set $ \{x:P_\epsilon(x)\le a\} $ for which the connected component containing the equilibrium, denoted by $ D_\epsilon $ satisfies
\begin{enumerate}
\item $ H(D_\epsilon,cl(S_f)\le\epsilon $,
\item $ P_\epsilon $ proves exponential decay over $ D_\epsilon $,
\end{enumerate}
where $ cl(S_f) $ is the closure of $ S_f $ and  $ H(A,B) $ denotes the Hausdorff distance between the sets $ A $ and $ B $.

The paper is organized as follows. Notation is introduced in Sec.~\ref{sec: notaiton}. A few basic definitions and preliminary lemmas are presented in Sec.~\ref{sec: Preliminaries}. The mathematical formulation of the problem and corresponding theorems and assumptions are stated in Sec.~\ref{sec :Problem Statement}. The main result of the paper is presented and proved in Sec.~\ref{sec: The Main Result}.
We give our proposed method for estimating the ROA and apply it to a numerical example in Sec.~\ref{sec: Numerical Example}. Finally, we conclude in Sec.~\ref{sec: Conclusion}.

\section{Notation}
\label{sec: notaiton}
The set of n-tuples of nonnegative natural numbers is denoted by $ \mathbb{N}^n $.  Let $ \R^+ $ and $ \R^- $ be the sets of positive and negative real numbers, respectively.
We denote the closed ball of radius $r\in\R^+$ centered at $c\in \R^n$ as $ B_{r}(c):=\{x\in\mathbb{R}^n\; :\;  \norm{x-c}_2\le r\}$ with the unit ball centered at the origin $B:=B_1(0)$.  For any subset, $D$, of a normed space, we use $D^{\mathrm{o}} $ to denote the interior of $D$ and $ \partial D $ to denote the boundary of $ D $ and $cl(D):= D\cup \partial D$ to denote the closure of $D$.

For operators $ g_i:X\rightarrow X $, we denote the composition  $ \Pi_i\; g_i:= g_1\; \mathrm{o}\dots\mathrm{o}\;g_m $. For any suitably differentiable function $f:\mathbb{R}^n\rightarrow\mathbb{R}^n $ and $ \alpha \in \mathbb{N}^n $, we adopt the multi-index differential notation
\[
D^\alpha f(x):= \frac{\partial^\alpha}{\partial x ^\alpha} f(x) = \prod_{i=1}^{n} \frac{\partial^{\alpha_i}}{\partial x_i ^{\alpha_i}} f(x),
\]
where, $ \partial^{0}/\partial x_i ^{0} f(x) := f. $ The gradient operator then is given by
\[
\nabla:= \begin{bmatrix}
	D^{(1,0,\dots,0)}\\\vdots\\ D^{(0,\dots,0,1)}
	\end{bmatrix}.
\]
For any $ k\in \mathbb{N} $ and $ \Omega\subset \mathbb{R}^n $, let  $ C_1^k(\Omega) $ be the set of all functions $  f:\Omega\rightarrow\mathbb{R}^n$ such that   $ D^\alpha f(x)$ is continuous for any $ \alpha \text{ with } ||\alpha||_1 := \sum_{j=1}^{n} \alpha_j\le k $. Let $ Z^n:=\{\alpha\in\mathbb{N}^n|\alpha_i\in\{0,1\},\; i=1,\dots,n\} $.

Finally, for any $f:D\rightarrow\mathbb{R}$ and $ a\in\mathbb{R} $, we use $ L(f,a):=\{x\in D | f(x)\le a\} $ to denote the $a-$ sub-level set of $f$.

\section{Basic Set Norms and Operations}
\label{sec: Preliminaries}

The technical contribution of this paper is to show that polynomial Lyapunov functions can be used to estimate the Region of Attraction arbitrarily well for suitably differential vector fields. This existence result requires approximation not just of the Lyapunov function and its derivatives, but the sub-level sets of the Lyapunov function as well. For this reason, we require a distance metric on sets for which we define convergence. The set distance we use is the Hausdorff metric which, for any two compact sets $ D_1,D_2\subset \mathbb{R}^n $, is defined as
\[
 H(D_1,D_2):=\max\{\zeta(D_1,D_2),\zeta (D_2,D_1)\},
 \]
where
\[
 \zeta(D_1,D_2) := \underset{x\in D_1}{\max} \;\underset{y\in D_2}{\min}||x-y|| .
 \]
Intuitively $ \zeta(D_1,D_2)$ is the farthest distance of any point in $D_1$ from the set $D_2$. The Hausdorff metric, then is the farthest distance from any point in either set to some point in the other set. Of course, if $D_1$ is a subset of $D_2$, then $\zeta(D_1,D_2) = 0 $ and in this case we would have $ H(D_1,D_2) = \zeta(D_2,D_1) $.

In the following lemma, we see that sequential subsets satisfy something akin to the triangle inequality in Hausdorff metric.
\begin{mylem}
\label{lem: Hausdorff}
Let $ X,Y $ and $ Z $ be compact subsets of $ \mathbb{R}^n $ such that $ X\,\subset\,Y\,\subset\,Z$. Then,
\[
H(X,Z)\ge \max\{H(X,Y),H(Y,Z)\}.
\]
\end{mylem}
\begin{proof}
Since $  X\,\subset\,Y\,\subset\,Z$, we have
\begin{align*}
 H(X,Y) = \zeta(Y,X) :&= \underset{y\in Y}{\max} \;\underset{x\in X}{\min}||x-y||\\
 &\le \underset{z\in Z}{\max} \;\underset{x\in X}{\min}||x-z||\\
 &=H(X,Z).
\end{align*}
Therefore, $ H(X,Y)\le H(X,Z) $. The proof of $ H(Y,Z)\le H(X,Z) $ is similar.
\end{proof}

\section{Definitions, Assumptions, and Converse Lyapunov Theory}
\label{sec :Problem Statement}
In this paper, we consider nonlinear differential equations of the form
\begin{equation}
	\label{eq: diff}
	\dot{x}(t)=f(x(t)),\quad x(0)=x_0,\quad t\in[0,\infty),
\end{equation}
where $ f:\mathbb{R}^n\rightarrow\mathbb{R}^n $ and $ f(0) = 0 $. For simplicity, in the following we will assume that the solution map for Eq.~\eqref{eq: diff} is well defined. That is, for any $x_0 \in \R^n$, there exists a unique function $g(x,t)$ such that $\partial_t g(x,t)=f(g(x,t))$ and $g(x,0)=x_0$.

\begin{mydef}
\label{def: Ass stable}
We say a set $ U\subset\mathbb{R}^n $ is \textit{asymptotically stable} for Eq.~\eqref{eq: diff} if
\begin{itemize}
\item $U$ contains a neighborhood of the origin.
\item For any $x \in U$, $g(x,t)\in\, U$ for all $t\ge 0$ and $\lim_{t\rightarrow\infty} g(x,t) =\, 0, $
\end{itemize}
\end{mydef}

\begin{mydef}
	\label{def: exponential}
We say that $ W\subset\mathbb{R}^n $ is an \textit{exponentially stable set} for Eq.~\eqref{eq: diff} if there exist $ \mu>0,\; \delta>0$ such that for any $x \in W$,
\begin{equation*}
	||g(x,t)||\le \mu ||x||\exp({-\delta\,t}).
\end{equation*}
\end{mydef}

The Region of Attraction of Eq.~\eqref{eq: diff} is defined as follows
\begin{mydef}
\label{def: ROA}
The \textit{Region of Attraction (ROA)} of the origin for Eq.~\eqref{eq: diff} is the asymptotically stable set $S$ such that for any asymptotically stable set $ U $, $U$ is contained in $S$. That is, $S$ is the union of all asymptotically stable sets, $S= \cup_{U \text{is AS}} U $. For convenience, we will henceforth denote the ROA for $f$ as $S_f$.
\end{mydef}
Note that ROA is an open set.
%

In this paper, we assume that $S_f$ exists and is bounded. Furthermore, without loss of generality, we assume $S_f$ is contained in the unit ball - i.e. $S_f \subset B$.
\begin{myass}
	\label{ass: general RA}
	$S_f$ exists and $S_f\subset B$.
\end{myass}

\noindent \textbf{Lyapunov Theorem}

\begin{mythm}
\label{theo:Lyapunov}
Let $V$ be a continuously differentiable function, $a,\beta,\gamma,\delta>0$, and $D$ be the connected component of $ L(V,a) $. Further suppose
			\begin{align*}
			\beta||x||^2&\le V(x)\le \gamma||x||^2,\\
			\nabla V(x)^T\,f(x)&\le -\delta ||x||^2,
			\end{align*}
			for all $ x\in D$. Then $D$ is an exponentially stable set for Eq.~\eqref{eq: diff}, as  in Definition~\ref{def: exponential}.
\end{mythm}

\noindent \textbf{Converse Lyapunov Results} There are two converse Lyapunov results which are of particular interest to this paper.

The first, Massera-type, result states that if $D^\alpha f$ is continuous and $f$ is exponentially stable on $D$, then there exists a Lyapunov function $V$ such that $D^\alpha V$ is continuous. Furthermore, there exists $\beta, \gamma, \delta >0$ such that
\[
\beta \norm{x}^2 \le V(x)\le \gamma \norm{x}^2,\qquad \nabla V(x)^T f(x)\le \gamma \norm{x}^2
\]

\begin{mythm}
	Consider ODE ~\eqref{eq: diff} and suppose $ f $ is $ k- $times continuously differentiable, for some $ k\in\N. $ Suppose there exists constants $ \lambda,\;\mu,\;\delta,\;r>0 $ such that
	\begin{align}
		||g(x,t)||&\le\mu||x||\exp(\delta t),\forall t\ge0,\quad\forall x\in B_r(0),\\
		||\nabla f||&\le\lambda, \quad\forall x\in B_{\mu r}(0).
	\end{align}
	Then there exist a $ k- $times continuously differentiable function $ W:\R^n\rightarrow \R $ and constants $ \alpha,\beta,\gamma,\mu>0 $ such that
	\begin{align}
		\alpha||x||^2\le W(x)\le\beta||x||^2,\quad\forall x\in B_r(0),\\
		\nabla W(x)^T f(x)\le-\gamma||x||^2,\quad\forall x\in B_r(0).
	\end{align}
\end{mythm}

The second is a maximal Lyapunov function result which says that there exists some Lyapunov function $V$ and a positive definite function $ \phi $ such  $\nabla V(x)^T f =-\phi(x) $ for all $x \in S_f$. Furthermore, the Lyapunov function is maximal in the sense that for any $y\in \partial S$, $ \underset{x\rightarrow y}{\lim}V(x)=+\infty$.

\begin{mythm}
	Consider ODE ~\eqref{eq: diff} and suppose $ f $ is Lipschitz continuous on $ S_f $. Then there exist a continuous function $ V:S_f\rightarrow\R^+\cup\{0\} $ and a positive definite function $ \phi $ such that
	\begin{align*}
		V(0) &= 0,\quad V(x)>0,\quad\forall x\in S_f\backslash\{0\},\\
		\dot{V}(x)& =-\phi(x),\quad\forall x\in S_f.\\
		\underset{x\rightarrow y}{\lim}V(x)& = \infty,\quad\forall x\in S_f,\;\forall y\in\partial_{S_f}.
	\end{align*}
\end{mythm}

A maximal Lyapunov function, $V$, has the advantage that for any desired accuracy, $\epsilon>0$, there exists a sublevel set $L(V,a) \subset S_f$ of $V$ with $H(S_f,L(V,a))\le \epsilon$ in the Haussdorf metric.

Under Assumption~\ref{ass: general RA}, from~\cite{VANNELLI198569} we can conclude that  if $ f $ is continuously differentiable, then there exists a continuously differentiable maximal Lyapunov function $ V:S_f\rightarrow\R^+ $ which decreases  over the trajectories of the system and satisfies
\[
 \underset{x\rightarrow y}{\lim}V(x)=+\infty,\quad\forall\;y\in \partial S.
\]
Furthermore $ V(x)$ proves exponential decay over each of the sub-level sets of $ V(x) $ and these sublevel sets can approximate $S_f$ as accurately as desired.

In the following assumption, we combine the maximal and quadratic results to simply assume that for any $\epsilon>0$, there exists a quadratically upper and lower bounded Lyapunov function, with quadratically upper bounded derivative and which is maximal. This slightly strong assumption can be relaxed at the cost of increasing the complexity of the proof.

\begin{myass}
	\label{ass: general LF}
	Let $ S $ be the ROA of Eq.~\eqref{eq: diff}. There exists a function $ V:S\rightarrow \mathbb{R} $ s.t.
	\begin{enumerate}
		\item  $ D^\alpha V \in C_1^2(S_f) ,\quad\forall\; \alpha \in Z^n $,
		\item $ \underset{x\rightarrow y}{\lim}V(x)=+\infty,\quad\forall\;y\in \partial S_f$,
		\item for all $ y\in S_f$:
		\begin{enumerate}
			\item $L(V,V(y))$ is compact.
			\item There exist  $ \beta_{y},\gamma_{y},\delta_{y}\in\mathbb{R}^+ $ s.t.
					\begin{align*}
						\beta_{y}||x||^2\le V(x)\le \gamma_{y}||x||^2,\quad&\forall\;x\in L(V,V(y)),\\
						\nabla V(x)^T\,f(x)\le -\delta_{y} ||x||^2,\quad&\forall\;x\in L(V,V(y)).
					\end{align*}

		\end{enumerate}
	\end{enumerate}

\end{myass}
Note that these assumptions are NOT mutually exclusive, as the upper and lower bounds only apply on a strict subset of the ROA. Furthermore, as indicated in the following lemma, these assumptions guarantee the existence of a decreasing Lyapunov function $V$ whose sub-level sets can approximate the ROA, $S_f$ for any desired level of accuracy in the Haussdorf metric.

\begin{mylem}
	\label{lem: approx for v}
	Suppose $ S_f $ is the ROA of Eq.~\eqref{eq: diff} as per Definition~\ref{def: ROA} and let $ V:S_f\rightarrow \R $ be any function satisfying the conditions in Assumptions~\ref{ass: general RA} and~\ref{ass: general LF}. For any $ \epsilon>0 $, there exists some $r>0$ such that
	\begin{align*}
	H(cl(S_f),L(V,r))\le \epsilon .
	\end{align*}
\end{mylem}
\vspace{0.1 in}

\begin{proof}
By definition, $ 0\in int(S_f)$. Now let $ S_\epsilon $ be any compact set such that $0 \in S_{\epsilon}$ and $S_\epsilon \subset S_f$ with $ H(S_\epsilon,cl(S))\le \epsilon $. Define
\[
x^*: = \underset{x\in S_\epsilon}{\arg \max}\;V(x),
\]
and $ r^*: = V(x^*)$. By the definition of $x^*$, $S_\epsilon \subset L(V,r^*)\subset S$. Therefore, if we let
	\begin{align*}
	X:=S_\epsilon,\;\; Y:=L(V,r^*),\;\; Z:=cl(S),
	\end{align*}
then $X \subset Y \subset Z$ and by Lemma~\ref{lem: Hausdorff},
\[
H(cl(S),L(V,r^*))=H(Z,Y)\le H(X,Z)=H(S_\epsilon,cl(S)) \le \epsilon.
\]
\end{proof}

In fact, under Assumptions~\ref{ass: general RA} and~\ref{ass: general LF}, a stronger notion of stability over the ROA can be verified. That is, let $ S_f $ be the ROA of Eq.~\eqref{eq: diff} and $ V:S\rightarrow\mathbb{R} $ be a function for which the conditions in Assumptions~\ref{ass: general RA} and~\ref{ass: general LF} are satisfied. Then Theorem~\ref{theo:Lyapunov} implies that  for any $ y\in S $, $ L(V,V(y)) $ is an exponentially stable set for Eq.~\eqref{eq: diff}. We also notice that a straightforward implication of the exponential stability of the sets $ L(V,V(y)) $ is that $ V(0)=0 $, $ V(x)>0,\;\forall x\in S\backslash\{0\} $ and the sets $ L(V,V(y)) $ are connected.

\section{The Main Result}
\label{sec: The Main Result}

We start this section by recalling the goal of the paper that is to determine  conditions that guarantee the sub-level sets of polynomial LFs can alone inner-approximate the ROA up to any desired accuracy. The main result shows that this guaranty can be provided under Assumptions~\ref{ass: general RA} and~\ref{ass: general LF}.


Before presenting the main result, we give a slight modification of a result in~\cite{4908942}, wherein it was shown that polynomial Lyapunov functions could approximate twice continuously differential Lyapunov functions with a quadratic upper bound on the error.

\begin{mylem}
\label{lem: pol aprox}
Suppose $S_f$ is the ROA of Eq.~\eqref{eq: diff} as in Definition~\ref{def: ROA} and $ V:S_f\rightarrow \R $ satisfies the conditions in Assumptions~\ref{ass: general RA} and~\ref{ass: general LF}. Then for any $ \epsilon>0 $ and $c>0 $, there exists a polynomial Lyapunov function $ P $ and positive scalers $ \beta,\gamma$ and $\delta $ s.t.
		\begin{align*}
		\beta||x||^2&\le P(x)\le \gamma||x||^2,\\
		\nabla P(x)^T\,f(x)&\le -\delta ||x||^2,\\
		|{P(x)-V(x)}|&\le \epsilon,
		\end{align*}
		for all $ x\in  L(V,c)$.
\end{mylem}
See the Appendix for a Proof.

\begin{mythm}
	\label{thm: main}
	Let $ S_f $ be the ROA of Eq.~\eqref{eq: diff} as in Definition~\ref{def: ROA} and $ V:S_f\rightarrow\R $ be any function for which the conditions in Assumptions~\ref{ass: general RA} and~\ref{ass: general LF} are satisfied. For any $ \epsilon>0 $, there exists a polynomial Lyapunov Function $ P $ and a sub-level set $ L(P,a) $ such that if we define $ D $ to be the connected component of $ L(P,a) $ containing the origin, then
	\begin{enumerate}
		\item $ H(D,cl(S_f))\le\epsilon $,
		\item  There exist scalers $\beta,\gamma,\delta\in\mathbb{R}^+ $ such that
			\begin{align*}
			\beta||x||^2&\le P(x)\le \gamma||x||^2,\\
			\nabla P(x)^T\,f(x)&\le -\delta ||x||^2,
			\end{align*}
			for all $ x\in D$.
	\end{enumerate}
\end{mythm}
This theorem states that $P$ can be used to prove stability on a set $D$ which is arbitrarily close to $S_f$ in the Haussdorf norm.
\vspace{0.1 in}

\begin{proof}
In this proof, we combine the fact that the sub-level sets of a maximal LF  can inner-approximate the ROA as in Lemma~\ref{lem: approx for v}, and the fact that any sufficiently smooth maximal Lyapunov function $ V $ can be approximated by polynomial LF  on each  of the sub-level sets $ L(V,a),\;\forall a>0 $ as in Lemma~\ref{lem: pol aprox},  to  show that	under Assumptions~\ref{ass: general RA} and~\ref{ass: general LF}, the connected components of the sub-level sets of polynomial LFs can alone inner-approximate the ROA, arbitrarily well.

Let $ V:S_f\rightarrow\mathbb{R} $ satisfy the conditions in Assumption~\ref{ass: general LF}. From Lemma~\ref{lem: approx for v}, we know that there exists a scalar $ r_1>0 $ such that
	\[
	H(L(V,r_1),cl(S_f))\le\epsilon.
	\]
	Now, let $ r_2>r_1 $. By Lemma~\ref{lem: pol aprox}, there exist scalars $ \beta,\gamma$, $\delta $ and a polynomial Lyapunov function $ P(x) $ approximating $ V(x) $ over the compact set $ L(V,r_2) $ such that
			\begin{align}
			\label{eq: temp6ProofTheo}
			|P(x)-V(x)|&\le\frac{r_2-r_1}{9} ,\quad\forall x\in  L(V,r_2),\\
			\label{eq: temp4ProofTheo}\beta||x||^2&\le P(x)\le \gamma||x||^2,\quad\forall x\in  L(V,r_2),\\
			\label{eq: temp5ProofTheo}\nabla P(x)^T\,f(x)&\le -\delta ||x||^2,\quad\forall x\in  L(V,r_2).
			\end{align}
			
			Now, suppose we can establish the existence of a sub-level set $ L(P,a)$ with connected component $ D $ where $D$ is compact and
			\begin{align}
			\label{eq: temp1ProofThoe1}
			 L(V,r_1)\subset D\subset L(V,r_2).
			\end{align}
			Then, since $ L(V,r_2)\subset S $ and $ H(L(V,r_1),cl(S))\le\epsilon $, if we let $ X:= L(V,r_2) $, $ Y:=D $ and $ Z:=cl(S) $, we have $ X\subset Y\subset Z $. Therefore, Lemma~\ref{lem: Hausdorff} implies  that $ H(D,cl(S))\le\epsilon $. Moreover, since $ D\subset L(V,r_2)$, Eq.~\eqref{eq: temp4ProofTheo} and~\eqref{eq: temp5ProofTheo} imply that 	
			\begin{align}
					\label{eq: temp2ProofTheo1}	\beta||x||^2&\le P(x)\le \gamma||x||^2,\quad\forall x\in D,\\
					\label{eq: temp3ProofTheo1}	\nabla P(x)^T\,f(x)&\le -\delta ||x||^2,\quad\forall x\in D,
			\end{align}
			as desired.	
					The remainder of the proof is dedicated to establishing the existence of such a sub-level set $ L(P,a) $ and connected component $ D $.

First, let
	 \begin{align*}
				r_m&:= \frac{r_2-r_1}{3} + r_1,\\
				\; x_m &:= \underset{x\in\partial L(V,r_m)}{\arg\min} P(x),\\
				 a&:=P(x_m),\\
				  E&:= L(V,r_2)\cap L(P,a),
		\end{align*}
		Note that $ r_1<r_m<r_2 $,  and the existence of $ x_m $ follows from the fact that $ L(V,r_m) $ is compact, as in Assumption~\ref{ass: general LF}, and $ P $ is continuous.
		
		If we define $ D $ as the connected component of $ L(P,a) $ such that $ 0\in D $, then we will show that $ E = D $ and $ D $ satisfies Eq.~\eqref{eq: temp1ProofThoe1} which completes the proof.
		In order to do so, first note that $ E\subset L(V,r_2) $ holds by definition. Next, we will in turn show that:
		\begin{enumerate}
				\item \label{proof: part2} $ E $ is connected and $ 0\in E $,
				\item\label{proof: part3} $L(V,r_1)\subset E$,
	   \item \label{proof: part4}$ E = D$.
		\end{enumerate}
							
			
			 \textbf{Proof of Part~\ref{proof: part2}}: In order to show that $ E $ is connected and $ 0\in E $, we will show that $ E $ is an exponentially stable set for Eq.~\eqref{eq: diff}, as in Definition~\ref{def: exponential}. Note that any exponentially stable set is connected and contains $ 0 $.
			 Under Assumption~\ref{ass: general LF}, $ L(V,r_2) $ is an exponentially stable set.  Now, since $ E\subset L(V,r_2) $ by definition, there exist $ \mu>0,\; \delta>0$ such that
			 \begin{align*}
			||\phi(x,t)||&\le \mu ||x||\mathrm{e}^{-\delta\,t},\quad\forall t\ge 0,
			 \end{align*}
			 \[
			 \lim_{t\rightarrow\infty} \phi(x,t) =\, 0,
			 \]
			 for any $ x\in E\subset L(V,r_2) $.
			  Therefore, in order to prove the exponential stability of $ E $, we only need to show that
				\begin{align}
				\label{eq: temp1}
				\phi(x,t)&\in E,\quad\forall x\in E,\quad \forall t\ge0.
				\end{align}

				 Again, we use the exponential stability of  $ L(V,r_2) $  and the fact that $ E\subset L(V,r_2) $ to write:
				 \begin{align}
				 \label{eq: proof2Temp1}
				 	\phi(x,t)&\in L(V,r_2),\quad\forall t\ge 0,\;\forall x\in E.
				 \end{align}
				 Now, based on Eq.~\eqref{eq: temp5ProofTheo}, we have
				\begin{align}
				\label{eq: proof2Temp3}
				\nabla P(x)^T\,f(x)<0,\;\forall x\in L(V,r_2)\backslash0.
				\end{align}

				Therefore, we can conclude from Eq.~\eqref{eq: proof2Temp1} and~\eqref{eq: proof2Temp3} that
				\begin{align*}
				& P(\phi(x,t))\\
				 &= P(x) + \int_{0}^{t}\nabla P(\phi(x,\tau))^T f(\phi(x,\tau))d\tau < P(x),
				\end{align*}
				for any $x\in L(V,r_2)\backslash0$.
				Hence, since $ E\subset L(P,a) $, we have
				 \begin{align}	
				 	P(\phi(x,t))\le P(x)\le a,\;\forall t\ge0,\; \forall x\in E.
				 \end{align}
				 Therefore,
				 \begin{align}
				 \label{eq: proof2Temp2}
				 \phi(x,t)\in L(P,a),\;\forall t\ge0,\; \forall x\in E.
				 \end{align}

				  Now, since $ E= L(V,r_2)\cap L(P,a)$, Eq.~\eqref{eq: temp1} follows from Eq.~\eqref{eq: proof2Temp1} and \eqref{eq: proof2Temp2}, as desired.

		\textbf{Proof of Part~\ref{proof: part3}:} We will use Lemma~\ref{lem: intersection} in the Appendix to show that $ 	L(V,r_1)\subset E.	$ Based on Assumption~\ref{ass: general LF} and Part~\ref{proof: part2} of the proof, the sets $L(V,r_1)$ and $ E $ both contain $ 0 $ and are compact and connected. In order to apply Lemma~\ref{lem: intersection}, we only need to show that
		
		\begin{enumerate}
		\item $L(V,r_1) \cap E\neq\emptyset,$
		
		\item $		\partial L(V,r_1) \cap \partial E = \emptyset,$
		
		\item $ 	E\not\subset int(L(V,r_1)).$
		\end{enumerate}

	 First, It is immediate that $ 0\in L(V,r_1)\cap E$ and hence
	 \[
	 E\cap L(V,r_1)) \neq \emptyset .
	 \]
	
	 Second, by contradiction, we will show that
	 $ \partial L(V,r_1) \cap \partial E = \emptyset $.
	 Suppose $ \partial L(V,r_1) \cap \partial E \neq \emptyset $. Therefore,
	
	  \[\exists\; x \in \partial L(V,r_1) \cap \partial E  .\]
	   Now, since $ E = L(V,r_2)\cap L(P,a) $, we have
	    \[
	 \partial E \subset \partial L(V,r_2) \cup \partial L(P,a).
	 \] Therefore, at least one of the following should hold:
	 \begin{align}
	  P(x) &= a \text{ and } V(x) = r_1 ,\label{eq: temp3}\hspace{0.7 in }\\
	  \text{ or }\hspace{.8 in}&\nonumber\\
	  V(x) &= r_m \text{ and } V(x) = r_1\label{eq: temp2} .\hspace{ 0.7 in}
	 \end{align}
	  Assertion~\eqref{eq: temp2} is impossible because $ r_1<r_m $. Therefore, Assertion~\eqref{eq: temp3} holds.  However, since $ \{x_m,x\}\subset L(V,r_1)\subset L(V,r_2) $, Eq.~\eqref{eq: temp6ProofTheo} implies that
	 \begin{align}
	 |P(x_m)-V(x_m)|&\le(r_2-r_1)/9,\label{eq: temp6 }\\
	 |P(x)-V(x)|&\le(r_2-r_1)/9\label{eq: temp7}.
	 \end{align}
	 Hence, since $ r_m = \frac{r_2-r_1}{3} + r_1 $, we use Assertion~\eqref{eq: temp3} and Eq.~\eqref{eq: temp7} to write
	   \begin{align}
	   \label{eq: temp8}
	   |a-r_1|&\le(r_m-r_1)/3.
	   \end{align}
	   and substitute $ P(x_m)=a $ and $ V(x_m) = r_m $ in  Assertion~\eqref{eq: temp6 } to write
	   \begin{align}
	   	|a-r_m|&\le (r_m-r_1)/3.\label{eq: temp9}
	   \end{align}
	   However, Eq.~\eqref{eq: temp8} and~\eqref{eq: temp9} by triangle inequality imply that
	   \[
	   |r_m - r_1| \le 2|r_m-r_1|/3,
	   \]
	   which is a contradiction, because $ r_1\neq r_m $.
	
	   Continuing the examination of the conditions in Lemma~\ref{lem: intersection}, finally we show that	$E\not\subset int(L(V,r_1))$.  Note that $ P(x_m)= a$ and $ V(x_m )= r_m\le r_2 $. Therefore, $x_m\in E $ and $ x_m\notin int(L(V,r_1)) $. Hence, $	E\not\subset int(L(V,r_1))$, as desired. Finally, $ L(V,r_1)\subset E$ follows from Lemma~\ref{lem: intersection}.

	  \textbf{Proof of Part~\ref{proof: part4}:} Let $ D $ be the connected component of $ L(P,a) $  such that $ 0\in D $. We will show that \[ E=D .\]
	
	  From Part~\ref{proof: part2} of the proof, we know that $ E $ is connected and $ 0\in E $. Moreover, $ E\subset L(P,a) $ by definition. Therefore, \[E\subset D .\]
	
	   Now, by contradiction,  we will show that $ D\subset E $. Suppose $ D\not\subset E= L(V,r_2) \cap L(P,a) $. Therefore, since $ D\subset L(P,a) $, we conclude that $ D\not\subset L(V,r_2) $. This means that
	   \[
	  \exists x\in D \; \text{ such that }\; x\notin L(V,r_2) .
	   \]
	   Connectedness of $ D $ implies that there exists a continuous function $\psi:[0,1]\rightarrow D $ such that $ \psi(0)=0 $ and $ \psi(1) = x $. Given the facts that $ L(V,r_2) $ is compact, $ 0\in L(V,r_2) $ and $ x\notin L(V,r_2) $, it can be shown that
	   \[
	    \exists t\in[0,1] \;\text{ such that } \;y:=\psi(t)\in\partial L(V,r_2).
	   \]
	    Therefore, $ y\in D $ and $ y\in\partial L(V,r_2) $. Hence,
	   \begin{align}
	   \label{eq: temp7ProofTheo}
	   V(y) &= r_2 \ge r_m + 3|r_m-r_1|,
	   \end{align}
	   and
	   \begin{align}
	   P(y) &\le a = P(x_m)\\
	   &\le V(x_m) + |r_m-r_1|\\
	   \label{eq: temp8ProofTheo}
	   & = r_m + |r_m-r_1|.
	   \end{align}
	   Finally, Eq.~\eqref{eq: temp7ProofTheo} and~\eqref{eq: temp8ProofTheo} imply that $ V(y)-P(y)\ge 2|r_m-r_1|= 2\frac{r_2-r_1}{9}  $, which is a contradiction. Therefore, $ D\subset E $, which implies that $ D = E $, as desired.
\end{proof}

\begin{cor}
$ P $ proves the exponential stability of $ D $ for Eq.~\eqref{eq: diff}, as  in Definition~\ref{def: exponential}.
\end{cor}
		
		Theorem~\ref{thm: main} shows that under Assumptions~\ref{ass: general RA} and~\ref{ass: general LF}, sub-level sets of polynomial Lyapunov functions can inner-approximate the ROA up to any desired accuracy.
		In order to illustrate the practical implication of Theorem~\ref{thm: main}, in the following section we will bring an example of an ODE with  a bounded ROA. For this ODE, we will show that  the sub-level sets of polynomial Lyapunov functions of  degree less than $ d $ can inner-approximate the ROA and these inner-approximations approach the true ROA as $ d $ increases.

	\section{A Proposed Method for Using SOS Programming to Approximate the ROA}
	\label{sec: Numerical Example}
	
	In this section, we propose a Lyapunov based approach to approximating the ROA of  polynomial ODEs based on the use of Sum of Squares (SOS) polynomials. Furthermore, we illustrate  the convergence of the proposed approach to the true ROA as applied to the van-der-Pol oscillator.
	
Consider the ordinary differential equation
	\begin{align}
	\label{eq: numericalExample1}
			\dot{x} = f(x),\quad\,x(0) = x_0,
	\end{align}
	where $ f:\R^n\rightarrow\R^n $ is a polynomial and $ f(0)=0 $. Denote by $ S_f $ the ROA of Eq.~\eqref{eq: numericalExample1} around the equilibrium $ 0 $ and suppose $ S_f $ is  nonempty and bounded. For any $ r\ge0 $, we represent the ball $B_r(0)$ as $\{x:u_r(x)\ge 0\}$ where $ u_r(x):= r^2-\sum_{i=1}^{n} x_i^2 $. If we can find a polynomial Lyapunov function $ P(x) $ and positive scalers $ \beta,\gamma $ and $ \delta $ such that
	\begin{align}
				\beta||x||^2&\le P(x)\le \gamma||x||^2,\label{eq: NE1}\\
				\nabla P(x)^T\,f(x)&\le -\delta ||x||^2,\label{eq: NE2}
				\end{align}
	for all $ x\in B_r(0)=\{x:u_r(x)\ge 0\}$, then for any sub-level set $ L(P,a)$ such that $\LARGE(P,a)\subset B_r(0) $, the connected component containing the origin is an inner-approximation to the ROA. The search for a polynomial Lyapunv function $ P $ satisfying Eq.~\eqref{eq: NE1} and~\eqref{eq: NE2} can be formulated using SOS programming. Specifically, for a fixed degree $d$, we have the polynomial variable $ P (x)$  and SOS polynomial variables $ s_1(x),\dots,s_6(x) $  of degree less than $ 2d $ with the constraint that
	\begin{align}
	P(x)-\beta\sum_{i=1}^{n}x_i^2 = s_1(x) + s_2(x) u_r(x),\\
	-P(x) + \gamma\sum_{i=1}^{n}x_i^2 = s_3(x) + s_4(x) u_r(x),\\
	-\nabla P(x)^Tf(x) - \delta\sum_{i=1}^{n}x_i^2 = s_5(x) + s_6(x) u_r(x).
	\end{align}
This form of SOS programming problem can be solved efficiently using such Matlab toolboxes as SOSTOOLS. For convenience, for given radius $r$ and degree $d$, we refer to this SOS program as $P=H(d,r)$, where $P$ is the feasible polynomial if $H$ is feasible and $P=\emptyset$ otherwise. We now propose the following two-step bisection-based approach to estimating the ROA using $H(d,r)$
\begin{enumerate}
  \item Initialize $r_{\max}$, $r_{\min}$. 
  \item Set $r=\frac{r_{\max}-r_{\min}}{2}$.
  \item If $H(d,r)$ is feasible, set $r_{\min}=r$, otherwise $r_{\max} =r$
  \item Goto step 2.
\end{enumerate}
The estimate of the ROA is then recovered from the last feasible $P=H(d,r)$ using an auxilliary SOS program to find the largest $a(d,r)$ such that $L(H(d,r),a(d,r))\subset B_r(0)$.

Now clearly, for any $d$, a necessary condition for the feasibility of $H(d,r)$ is $ S_f\not\subset B_r (0) $. 
Now define 
\[
r^* = \sup_r  r \text{ such that } S_f\not\subset B_r (0) 
\]
Now, based on the results of this paper and numerical experimentation, we propose the conjecture that as $r \rightarrow r^*$, the polynomial $P$ must approximate some maximal Lyapunov function arbitrarily well in some neighborhood of the boundary of the ROA and this convergence can be extended to the level sets of the Lyapunov function. Furthermore, since the ROA is compact, we conjecture that this approximation can be extended to the entire ROA. Or, in other words
\[
\lim_{\substack{r\rightarrow r^*\\d \rightarrow \infty}} L(H(d,r),a(d,r)) = ROA.
\]
Note that although the results of the paper do not establish this convergence, they are \textit{necessary} for the conjecture to be true.

\subsection{Numerical Illustration}	
In this subsection, we show the apparent convergence of the proposed method.

	\textbf{Example:} 	Consider the Van der Pol oscillator in reverse time defined as
		\begin{align}
		\label{eq: vanderpol}
		\dot{x_1}&= -x_2,\\
		\dot{x_2} &= x_1+x_2(x_1^2-1).\nonumber
		\end{align}
	
    We applied the proposed method for degrees  $ d=4 $, $ 6 $ and $ 8 $.  Figure~\ref{fig} shows the corresponding recovered inner-approximate ROA compared to the true ROA as defined by forward-time numerical integration of Eq.~\eqref{eq: vanderpol}.  In addition, the maximal radius on which $H(d,r)$ is feasible is indicated for $ d=4,\; 6,\;8 $ by the corresponding dashed circle. 

	\begin{figure}[h]\centering
		
		\hspace{0in}\includegraphics[scale=0.4]{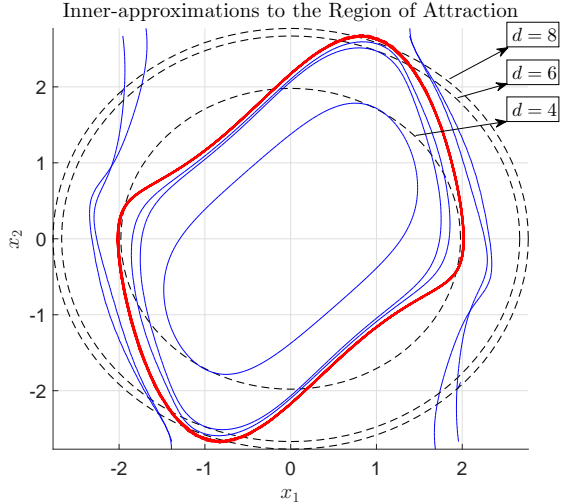}
		\caption{  Inner-approximations to the ROA of Eq.~\eqref{eq: vanderpol} obtained by solving the associated SDP for polynomial variables with degree bound  $ 2d $. }
		\label{fig}
	\end{figure}

	\section{Conclusion}
	\label{sec: Conclusion}
In this paper, we have proposed an SOS-based method for estimating the Region of Attraction of nonlinear ODEs, the conjectured convergence of which relies on the assumption that polynomial Lyapunov functions can estimate the true ROA arbitrarily well. To verify this assumption, we have presented sufficient conditions which guarantee that the ROA can be inner-approximated by the sub-level sets of polynomial Lyapunov Functions arbitrarily well in sense of Hausdorff distance. The main result of the paper as presented in Theorem~\ref{thm: main} is that if the ROA is bounded and there exists an n-times continuously differentiable maximal LF satisfying the conditions in Assumption~\ref{ass: general LF}, then for any scaler $ \epsilon>0 $, there exists a polynomial LF, $ P $, and a sub-level set, $ L(P,a) $, such that if we define $ D $ as the connected component of $ L(P,a) $ that contains the equilibrium, then $ P $ proves exponential stability of the ODE on the set $ D $ and $ D $ is within the Hausdorff $ \epsilon $ distance of the ROA. In order to demonstrate convergence of our proposed approach consistent with the existence results, we applied the methodology to a $ 2- $dimensional polynomial ODE. Work is ongoing to show that as the domain approaches the ROA in the Hausdorff metric, any polynomial Lyapunov function decreasing on that domain \textit{must} provide an asymptotically accurate estimate of the ROA.

\appendix
In this appendix, we provide a proof for Lemma~\ref{lem: pol aprox}, that was used in the proof of Theorem~\ref{thm: main}.

\textit{Lemma~\ref{lem: pol aprox}:}
Suppose $S_f$ is the ROA of Eq.~\eqref{eq: diff} as in Definition~\ref{def: ROA} and $ V:S_f\rightarrow \R $ satisfies the conditions in Assumptions~\ref{ass: general RA} and~\ref{ass: general LF}. Then for any $ \epsilon>0 $ and $c>0 $, there exists a polynomial Lyapunov function $ P $ and positive scalers $ \beta,\gamma$ and $\delta $ s.t.
		\begin{align*}
		\beta||x||^2&\le P(x)\le \gamma||x||^2,\\
		\nabla P(x)^T\,f(x)&\le -\delta ||x||^2,\\
		|{P(x)-V(x)}|&\le \epsilon,
		\end{align*}
		for all $ x\in  L(V,c)$.

\begin{proof}
By assumption, there exist $ \beta_0,\gamma_0,\delta_0>0$ such that
		 \begin{align*}
		 \beta_0||x||^2&\le V(x)\le \gamma_0||x||^2,\\
		 \nabla V(x)^T\,f(x)&\le -\delta_0 ||x||^2,
		 \end{align*}
		 for all $ x\in L(V,c) $.

Now choose $ \beta,\gamma $ and $ \delta $ such that
\[
			\beta<\beta_0,\quad \gamma>\gamma_0,\quad \delta<\delta_0 .
\]
Given these $ \beta,\gamma $ and $ \delta $, we define
		\begin{align*}
				b &:= \underset{x\in B}{\max}||f(x)||_\infty,\\
				d &:= \min \{\beta_0-\beta,\;\gamma-\gamma_0\,,(\delta_0-\delta)\,/nb,\; \epsilon\},
		\end{align*}
By Theorem 8 in~\cite{4908942}, there exists a polynomial $P(x)$ such that
		\begin{align*}
			\left|\frac{P(x)-V(x)}{x^Tx}\right|&\le d,\\
			\left|\frac{\frac{\partial P(x)}{\partial x_i}-\frac{\partial V(x)}{\partial x_i}}{x^Tx}\right|&\le d,\quad\forall i=1,\dots,n,
		\end{align*}
			for all $ x\in  L(V,c)$. Expanding, we have
			\begin{align*}
			P(x) &= V(x) + \frac{P(x)-V(x)}{x^Tx}x^Tx\\
				&\ge (\beta_0-d)x^Tx\ge \beta x^Tx.
			\end{align*}
and
			\begin{align*}
				P(x) &= V(x) + \frac{P(x)-V(x)}{x^Tx}x^Tx\\
					&\le (\gamma_0+d)x^Tx\le \gamma x^Tx.
			\end{align*}
			 Finally:
			 \begin{align*}
			 \nabla P(x)^T\,f(x)& =\nabla V(x)^T\,f(x) + \frac{\nabla (P(x)-V(x))^Tf(x)}{x^Tx}x^Tx\\
			 &\le (-\delta_0+ n\,d\,b)x^Tx \le -\delta x^Tx.
			 \end{align*}
			 Moreover, we have
			\[|P(x)-V(x)|\le|\frac{P(x)-V(x)}{x^Tx}|\le d\le \epsilon,\;\forall x\in  L(v,c) ,\] as desired.
\end{proof}

\begin{mylem}
	
	\label{lem: intersection}
	Let $ A,B\subset\mathbb{R}^n $ be nonempty, connected and compact and $ \partial A \cap \partial B = \emptyset $. Then, one and only one of the followings hold:
	\begin{itemize}
		\item $ A\subset  B^{\mathrm{o}} $,
		\item $ B\subset A^{\mathrm{o}} $,
		\item $ A\cap B =\emptyset $.
	\end{itemize}
\end{mylem}

\bibliographystyle{ieeetr}
\bibliography{ROA}
\end{document}